\documentclass[12pt,oneside,a4paper,reqno]{amsart}

\usepackage{amssymb}
\usepackage{amsmath}
\usepackage{amsthm}
\usepackage{amscd}
\usepackage[all]{xy}
\usepackage{longtable}
\usepackage{mathrsfs}
\usepackage[dvipdfmx]{xcolor}
\usepackage[dvipdfmx]{pict2e}
\usepackage[dvipdfmx]{graphicx}

\usepackage{comment}
\usepackage{todonotes}

\usepackage[dvipdfmx]{hyperref}
\usepackage{hyperref}
\hypersetup{
	colorlinks=true,
	linkcolor=red,
	citecolor=blue}
	\usepackage[utf8]{inputenc}
\usepackage[T1]{fontenc}
\usepackage{color}

\usepackage{extarrows}
\usepackage{tikz}
\usetikzlibrary{cd}

\usepackage{geometry}
\theoremstyle{plain}
\newtheorem{theorem}{Theorem}[section]

\newtheorem{corollary}[theorem]{Corollary}
\newtheorem{proposition}[theorem]{Proposition}
\newtheorem{remark}[theorem]{Remark}

\theoremstyle{definition}
\newtheorem{definition}[theorem]{Definition}

\newtheorem{conjecture}[theorem]{Conjecture}

\newtheorem{remark-theorem}[theorem]{Remark-Theorem}
\newtheorem{counterexample}[theorem]{Counterexample}

\newcommand{\kah}{K\"{a}hler }
\newcommand{\idd}{i\partial\overline{\partial}}
\newcommand{\dbar}{\overline{\partial}}

\newcommand{\cal}[1]{\mathcal{#1}}
\newcommand{\bb}[1]{\mathbb{#1}}
\newcommand{\scr}[1]{\mathscr{#1}}
\newcommand{\rom}[1]{\mathrm{#1}}

\newcommand{\xs}{X_{sing}}
\newcommand{\reg}{X_{reg}}

\newcommand{\tx}{\widetilde{X}}
\newcommand{\hx}{\widehat{X}}

\newcommand{\tl}[1]{\widetilde{#1}}
\newcommand{\wh}[1]{\widehat{#1}}

\newcommand{\tv}{\widetilde{V}}
\newcommand{\hv}{\widehat{V}}

\newcommand{\exc}{\rom{Exc}}

\newcommand{\hldea}{h^{\cal{L}}_{\delta,\varepsilon,\alpha}}

\newcommand{\iO}[1]{i\Theta_{#1}}


\begin{document}
\title
[Steenbrink Vanishing Theorem for Big Line Bundles] 
{Steenbrink Vanishing Theorem for Big Line Bundles}
\author{Yuta Watanabe}
\address{Department of Mathematics, Faculty of Science and Engineering, Chuo University.
1-13-27 Kasuga, Bunkyo-ku, Tokyo 112-8551, Japan}
\email{{\tt wyuta.math@gmail.com}, {\tt wyuta@math.chuo-u.ac.jp}}

\begin{abstract}
    In this paper, we generalize the Steenbrink vanishing theorem for ample line bundles on complex projective varieties by extending it to big line bundles on compact complex spaces with multiplier ideal sheaves.
\end{abstract}


\maketitle

\vspace{-5mm}




\section{Introduction}

Vanishing theorems for cohomology groups are a central topic in complex geometry and powerful tools in algebraic geometry and several complex variables.
For example, the existence of global holomorphic sections obtained from cohomology vanishing has many applications.
The Kodaira vanishing theorem \cite{Kod53} has since been generalized in various directions through numerous studies 
(see \cite{AN54,Ram72,Nor78,Bog79,Bog80,Kaw82,Vie82,SJ85,Ste85,Nad89,HLWY23,Wat23,Wat26a,Wat26b,LMNWZ25,MQZ26}). 
These developments have played an important role in a wide range of areas in complex geometry and have become indispensable in the classification theory of higher-dimensional projective algebraic varieties.

The Kodaira--Akizuki--Nakano vanishing theorem \cite{Kod53,AN54} asserts that if $X$ is a compact complex manifold of dimension $n$ and $A\longrightarrow X$ is an ample line bundle, then 
\begin{center}
    $H^q(X,\Omega_X^p\otimes A)=0$ \quad for any \quad $p+q>n$.
\end{center}
Extending the Kodaira--Akizuki--Nakano vanishing theorem from ample to big line bundles is a natural problem.
As indicated by the vanishing theorems of Nadel, Bogomolov, and Kawamata--Viehweg described below, such an extension is highly significant. 
Unlike the Kodaira--Akizuki--Nakano vanishing theorem, which extends to bidegrees $(p,q)$ with $p+q>n$, 
the Nadel vanishing theorem \cite{Nad89} provides an optimal extension of the Kodaira vanishing theorem from ample to big line bundles by twisting with multiplier ideal sheaves.
The Nadel vanishing theorem has a wide range of applications and plays an important role in complex geometry; it also implies the Kawamata--Viehweg vanishing theorem for nef and big line bundles \cite{Kaw82,Vie82}. 
Furthermore, the Nadel vanishing theorem was recently extended to the bidegree $(p,n)$ case in \cite{Wat23}.

In summary, if $X$ is projective and $L$ is big, equivalently, if $L$ admits a singular Hermitian metric $h$ with a strictly positive curvature current (see \cite{Dem90}), then
\begin{center}
    $H^q(X,\Omega_X^n\otimes L\otimes\scr{I}(h))=0$ \quad for any \, $q>0$ \quad \cite{Nad89}, \\
    $H^n(X,\Omega_X^p\otimes L\otimes\scr{I}(h))=0$ \quad for any \, $p>0$ \quad \cite{Wat23}.
\end{center}
The Nadel vanishing theorem in the first line can be extended to complex spaces; in particular, the projectivity assumption can be removed \cite{Wat26b} by using the vanishing theorem for higher direct image sheaves.
The second vanishing result above is a version of the Bogomolov vanishing theorem \cite{Bog79,Bog80} involving the multiplier ideal sheaf $\scr{I}(h)$, and more recently, further generalizations of this result have been obtained in \cite{LMNWZ25,Wat26a,MQZ26}.
However, the above vanishing results cannot be extended to the same bidegrees as in the Kodaira--Akizuki--Nakano vanishing theorem due to Ramanujam's counterexample \cite{Ram72} (see also \cite[Remark 2.10]{Wat23}).
On the other hand, the Steenbrink vanishing theorem for ample line bundles is given as follows.

\begin{theorem}[{Steenbrink, \cite[Theorem 2]{Ste85}}]\label{Theorem: Steenbrink vanishing [Ste85]}
    Let $X$ be a complex projective variety of dimension $n$, let $\Sigma\subset X$ be a closed subset such that $X\setminus\Sigma$ is nonsingular, let $A$ be an ample line bundle on $X$, and let $\mu:\tx\longrightarrow X$ be a proper birational mapping such that $\tx$ is nonsingular, $E:=\mu^{-1}(\Sigma)$ is a simple normal crossing divisor on $\tx$, 
    and $\mu$ induces an isomorphism $\tx\setminus E\cong X\setminus\Sigma$. Then, we have the following vanishing 
    
    \vspace{2mm}
    \begin{tabular}{ll}
    $(a)$ \quad $H^q(\tx,\Omega^p_{\tx}(\log E)\otimes\cal{O}_{\tx}(-E)\otimes\mu^*A)=0$ & \quad for any \quad $p+q>n$, \\[2mm]
    $(b)$ \quad $R^q\mu_*\big(\Omega^p_{\tx}(\log E)\otimes\cal{O}_{\tx}(-E)\big)=0$ & \quad for any \quad $p+q>n$.
    \end{tabular}
\end{theorem}

Although $\mu^*A$ is not necessarily ample, it is big. This suggests that logarithmic-type vanishing theorems for big line bundles may hold in the same bidegrees as the Kodaira--Akizuki--Nakano vanishing theorem. 
Therefore, we attempt to extend the ampleness assumption in the Steenbrink vanishing theorem to the case of big line bundles.

Logarithmic-type vanishing theorems are recalled below.
Let $X$ be a projective manifold and $D$ be a simple normal crossing divisor on $X$. For an ample line bundle $A$ on $X$, the vanishing result
\begin{center}
    $H^q(X,\Omega_X^p(\log D)\otimes A)=0$ \quad for any \quad $p+q>n$,
\end{center}
was obtained in \cite{Nor78} using analytic methods. Recently, a more refined analogous result concerning positivity was obtained in \cite{HLWY23}. 
As a direct consequence of \cite[Theorem 1.1]{HLWY23}, it follows that 
\begin{center}
    $H^q(X,\Omega_X^p(\log D)\otimes\cal{O}_X(-D)\otimes A)=0$ \quad for any \quad $p+q>n$.
\end{center}
More recently, this result has been extended to big line bundles as follows.

\begin{theorem}[{\cite[Corollary 1.3]{Wat26a}}]\label{Theorem: logarithmic vanishing for big line bundle in [Wat26a]}
    Let $X$ be a compact \kah manifold of dimension $n$, let $D$ be a simple normal crossing divisor on $X$, and let $L\longrightarrow X$ be a holomorphic line bundle. 
    If $L$ is big, that is, if $L$ admits a singular Hermitian metric $h$ with a strictly positive curvature current on $X$, then we have the following vanishing 
    \begin{align*}
        H^q(X,\Omega^n_X(\log D)\otimes\cal{O}_X(-D)\otimes L\otimes\scr{I}(h))=0& \qquad\text{for any}\quad q>0, \\ 
        H^n(X,\Omega^p_X(\log D)\otimes\cal{O}_X(-D)\otimes L\otimes\scr{I}(h))=0& \qquad\text{for any}\quad p>0.
    \end{align*}
\end{theorem}

Note that the vanishing in the first line follows immediately from the Nadel vanishing theorem, since $\Omega^n_X(\log D)=K_X\otimes\cal{O}_X(D)$. 
However, in general, Theorem \ref{Theorem: logarithmic vanishing for big line bundle in [Wat26a]} also cannot be extended to the same bidegrees as the Kodaira--Akizuki--Nakano vanishing theorem, due to the non-vanishing result induced by Ramanujam's counterexample (see Remark \ref{Remark: non-vanishing result}).
Therefore, assuming only bigness is not sufficient to obtain such an extension. 
Motivated by the observation that the non-ample behavior of a big line bundle can be partially removed by twisting with a multiplier ideal sheaf, we establish the following main theorem by taking suitable blow-ups.

\begin{theorem}\label{Theorem: Steenbrink vanishing for big in Introduction}
    Let $X$ be a compact complex space of pure dimension $n$, let $\mu:\hx\longrightarrow X$ be a resolusion of singularities, and let $L\longrightarrow X$ be a holomorphic line bundle.
    If $L$ is big, then $X$ is Moishezon and there exists a singular Hermitian metric $h$ on $L$ such that 
    $\mu^*h$ has algebraic singularities on $\hx\setminus\exc(\mu)$ and a strictly positive curvature current on $\hx$, where $\exc(\mu)$ is the $\mu$-exceptional divisor. 
    For a log resolusion $\pi:\tx\longrightarrow\hx$ of the singular locus $Z$ of $\mu^*h$, the following Steenbrink-type vanishing holds: 
    \begin{align*}
        H^q(\tx,\Omega^p_{\tx}(\log E)\otimes\cal{O}_{\tx}(-E)\otimes\tl{\pi}^*L\otimes\scr{I}(\tl{\pi}^*h))=0,
    \end{align*}
    for any $p+q>n$, where $\tl{\pi}:=\mu\circ\pi:\tx\longrightarrow X$, and $E:=\pi^{-1}(Z)_{red}$ is a simple normal crossing divisor. 
    Furthermore, if $X$ is normal, then $h$ is singular positive.
\end{theorem}

Theorem \ref{Theorem: Steenbrink vanishing for big in Introduction} is obtained by applying \cite[Theorem 1.1]{HLWY23} to the following positivity result for $\bb{Q}$-line bundles.
In particular, unlike the Steenbrink vanishing theorem, the projectivity of $X$ is not required.
More generally, a result corresponding to Theorem \ref{Theorem: Steenbrink vanishing for big in Introduction} can be proved for singular positive line bundles on relatively compact weakly pseudoconvex complex spaces (= Theorem \ref{Theorem: Steenbrink vanishing for singular positive in not Introduction}).

\begin{theorem}\label{Theorem: positivity of Q-line bundle in Introduction}
    Let $X$ be a compact complex manifold and $L\longrightarrow X$ be a holomorphic line bundle. 
    If $L$ is big, that is, if $L$ admits a singular Hermitian metric $h$ with algebraic singularities and a strictly positive curvature current on $X$, then for a log resolusion $\pi:\tx\longrightarrow\hx$ of the singular locus $Z$ of $h$, 
    there exist nonnegative rational numbers $\delta_j\in\![\upsilon_{E_j}(h)-\lfloor\upsilon_{E_j}(h)\rfloor,\!1)\cap\bb{Q}\subset[0,1)$ such that the following $\bb{Q}$-line bundle
    \begin{align*}
        \pi^*L\otimes\scr{I}(\pi^*h)\otimes\cal{O}_{\tx}\Big(-\sum_{j\in J}\delta_jE_j\Big)
    \end{align*}
    is positive on $\tx$, where $\sum_{j\in J}E_j:=\pi^{-1}(Z)_{red}$ is a simple normal crossing divisor 
    and $\upsilon_{E_j}(h):=\nu(\pi^*h,E_j)$ denotes the divisorial Lelong number of $h$ along $E_j$.
\end{theorem}

Finally, analogous to the Kawamata--Viehweg vanishing theorem, the Steenbrink vanishing theorem can be extended to nef and big line bundles.

\begin{theorem}\label{Theorem: Steenbrink vanishing for nef and big}
    Let $X$ be a compact complex space of pure dimension $n$, let $\mu:\hx\longrightarrow X$ be a resolusion of singularities, and let $L\longrightarrow X$ be a holomorphic line bundle.
    If $L$ is nef and big, then for a log resolusion $\pi:\tx\longrightarrow\hx$ of the singular locus $Z$ of a singular Hermitian metric $\wh{h}$ on $\mu^*L$, the following Steenbrink-type vanishing holds:
    \begin{align*}
        H^q(\tx,\Omega^p_{\tx}(\log E)\otimes\cal{O}_{\tx}(-E)\otimes\tl{\pi}^*L)=0,
    \end{align*}
    for any $p+q>n$, where $\tl{\pi}:=\mu\circ\pi:\tx\longrightarrow X$ and $E:=\pi^{-1}(Z)_{red}$ is a simple normal crossing divisor. 
    In particular, $\wh{h}$ can be chosen to have algebraic singularities, to admit a strictly positive curvature current, and to satisfy $\scr{I}(\wh{h})=\cal{O}_{\hx}$ on $\hx$.
\end{theorem}

\section{Preliminaries}

\subsection{Plurisubharmonicity and singular Hermitian metrics of line bundles on complex spaces}

Let $X$ be a complex space. 

\begin{definition}[{\cite[Chapter\,V, Definition\,1.4]{GPR94}}]
    A function $\varphi:X\longrightarrow[-\infty,+\infty)$ is called (resp. \textit{strictly}) \textit{plurisubharmonic} if for any $x\in X$ 
    there exist an open neighborhood $U$ admitting a closed holomorphic embedding $\iota_U\!:\!U\!\hookrightarrow \!V\!\subset\bb{C}^N$, here $V$ is an open subset, 
    and a (resp. strictly) plurisubharmonic function $\widetilde{\varphi}$ on $V$ 
    such that $\varphi|_U=\widetilde{\varphi}\circ\iota_U$.
\end{definition}

A function is said to be \textit{quasi}-\textit{plurisubharmonic} if it can be written locally as the sum of a smooth function and a plurisubharmonic function.
Let $L$ be a holomorphic line bundle. 
For any trivialization $\tau:L|_U\overset{\simeq}{\longrightarrow}U\times\bb{C}$, a Hermitian metric $h$ on $L$ can be expressed as $|\xi|_h=|\tau(\xi)|e^{-\varphi(x)}$, $x\in U$, $\xi\in L_x$, using a function $\varphi$ on $U$. 
The function $\varphi:U\longrightarrow\bb{R}$ is called the \textit{weight function of} $h$ \textit{with respect to the trivialization} $\tau$.

\begin{definition}[{\cite[Definition 2.6]{Wat26b}}]\label{Definition: singular Hermitian metrics on cpx sp}
    Let $X$ be a complex space and $L\longrightarrow X$ be a holomorphic line bundle. 
    We say that $h$ is a singular Hermitian metric on $L$ if for any smooth Hermitian metric $h_0$ on $L$, there exists a locally integrable function $\varphi$ on $X$, i.e., $\varphi\in L^1_{loc}(X)$, such that $h=h_0e^{-2\varphi}$ on $X$.
\end{definition}

\begin{definition}[{\cite[Definition 2.7]{Wat26b}}]\label{Definition: singular positivity on cpx sp}
    Let $X$ be a complex space. 
    We say that a singular Hermitian metric $h$ on $L$ is \textit{singular} \textit{positive} (resp. \textit{singular} \textit{semi}-\textit{positive}) if the weight function of $h$ with respect to any trivialization coincides with a strictly plurisubharmonic (resp. plurisubharmonic) function almost everywhere. 
\end{definition}

When $X$ is smooth, Definitions \ref{Definition: singular Hermitian metrics on cpx sp} and \ref{Definition: singular positivity on cpx sp} coincide with the existing definitions (see \cite[Chapter 3]{Dem12}, \cite{Wat25a,Wat25b}).

\subsection{Multiplier ideal sheaves, Lelong numbers and $\bb{R}$-line bundle}

In this subsection, we consider the smooth case, and let $X$ be a complex manifold 
and $L$ be a holomorphic line bundle with a singular Hermitian metric $h$ whose global weight $\varphi$ is quasi-plurisubharmonic, i.e., $h=h_0e^{-2\varphi}$ for some smooth Hermitian metric $h_0$ on $L$.

\begin{definition}[{\cite[Definition 2.3.9]{MM07}}]
    A real function $\varphi:X\longrightarrow[-\infty,+\infty)$ is said to have \textit{analytic} \textit{singularities} if $\varphi$ has locally the form 
    \begin{align*}
        \varphi=\frac{c}{2}\log\sum_{j\in J}|f_j|^2+\psi,
    \end{align*} 
    where $J$ is at most countable, $f_j$ are non-vanishing holomorphic functions, $\psi$ is a locally bounded function and $c\in\bb{R}_{>0}$.
    Here, the singular support of $\varphi$ and $\idd\varphi$ is an analytic subset.
    If $c\in\bb{Q}_{>0}$, then we furthermore say that $\varphi$ has \textit{algebraic} \textit{singularities}.
\end{definition}

Let $\varphi$ be a quasi-plurisubharmonic function on $X$.
We define the \textit{multiplier} \textit{ideal} \textit{sheaf} to be the ideal subsheaf $\mathscr{I}(\varphi)\subset\mathcal{O}_X$ of germs of holomorphic functions $f\in\mathcal{O}_x$ such that $|f|^2e^{-\varphi}$ is locally integrable near $x\in X$.
The multiplier ideal sheaf of $h$ is defined by $\scr{I}(h):=\scr{I}(\varphi)$.
The \textit{Lelong number} of $\varphi$ is defined by
\begin{align*}
    \nu(\varphi,x):=\liminf_{z\to x}\frac{\varphi(z)}{\log|z-x|}
\end{align*}
for some coordinate $(z_1,\ldots,z_n)$ around $x\in X$. 
The Lelong number of $h$ is defined by $\nu(h,x):=\nu(-\log h,x)/2=\nu(\varphi,x)$. 
For the relationship between the Lelong number of $\varphi$ and the integrability of $e^{-2\varphi}$, the following important result obtained by Skoda (see \cite[Lemma 5.6]{Dem12}) is known; 
If $\nu(\varphi,x)<1$ then $e^{-2\varphi}$ is integrable around $x$. From this, particularly if $\nu(h,x)<1$ then $\scr{I}(h)=\cal{O}_{X,x}$ immediately.

\begin{definition}
    Let $\varphi$ be a quasi-plurisubharmonic function on a complex manifold $X$.
    Let $\pi:Y\longrightarrow X$ be a proper modification of complex manifolds, and let $E$ be a prime divisor on $Y$. We define the \textit{divisorial Lelong number} of $\varphi$ along $E$ by 
    \begin{align*}
        \upsilon_E(\varphi):=\nu(\pi^*\varphi,E):=\inf_{y\in E}\nu(\pi^*\varphi,y), 
    \end{align*}
    and this equals the common value of $\nu(\pi^*\varphi,\eta_E)$ at a general point $\eta_E$ of $E$ (see \cite[Lemma 2.17]{Dem12}). 
    The divisorial Lelong number of $h$ along $E$ is defined by $\upsilon_E(h):=\upsilon_E(\varphi)$.
\end{definition}

We introduce the notion of $\bb{R}$-line bundles and their positivity. 
An $\bb{R}$-line bundle $L=\sum^k_{j=1}a_jL_j$ is a finite sum with some real numbers $a_1,\ldots,a_k$ and holomorphic line bundles $L_1,\ldots,L_k$. 
We say that $L=\sum^k_{j=1}a_jL_j$ is \textit{positive} on $X$ if there exist smooth Hermitian metrics $h_1,\ldots,h_k$ on $L_1,\ldots,L_k$ such that the curvature of the induced metric $h:=\prod^k_{j=1}h_j^{a_j}$ on $L$, which is explicity given by 
\begin{align*}
    \iO{L,h}=i\sum_{1\leq j\leq k}a_j\Theta_{L_j,h_j}, 
\end{align*}
is positive on $X$. Similarly, $\bb{Q}$-line bundles and their positivity are defined by taking the coefficients $a_j$ to be rational numbers.

\subsection{Positivity of curvature currents and bigness on compact space}

Singular positivity admits the following reformulation in terms of curvature currents. 

\begin{proposition}[{\cite[Proposition 3.1]{Wat26b}}]\label{Proposition: curvature condition of sHm}
    Let $X$ be a complex space of pure dimension and $L\longrightarrow X$ be a holomorphic line bundle with a singular Hermitian metric $h$. 
    We obtain the following relationship.
    \begin{itemize}
        \item If $h$ is singular semi-positive, then $\iO{L,h}\geq0$ on $X$ in the sense of currents. 
        \item If $h$ is singular positive, then for any Hermitian metric $\omega$ on $X$, there exists a positive smooth function $\varepsilon:X\longrightarrow\bb{R}_{>0}$ such that $\iO{L,h}\geq\varepsilon\omega$ on $X$ in the sense of currents. 
    \end{itemize}
    Furthermore, the converse also holds if $X$ is normal, or if any weight function is locally bounded from above in a neighborhood of $\xs$ and $X$ is locally irreducible.
\end{proposition}

Let $X$ be a complex space and $L\longrightarrow X$ be a holomorphic line bundle. 
We define $L$ to be \textit{big} if the Kodaira-Iitaka dimension of $L$ is maximal on each irreducible component. 
A compact complex space $X$ is said to be \textit{Moishezon} if the algebraic dimension of each irreducible component $X_j$ is maximal (cf. \cite{Moi66}).  
If $L$ is big, then $X$ is Moishezon, and every Moishezon space is bimeromorphic to a projective manifold (see \cite{Moi66}).
Thus, since a compact Moishezon space contains sufficiently many curves, we define a line bundle to be \textit{nef} as usual in terms of the non-negativity of its intersection numbers.
Furthermore, nefness is preserved under pullbacks by resolution of singularities.

When $X$ is a projective manifold, Demailly proved that a line bundle is big (resp. pseudo-effective) if and only if it admits a singular positive (resp. singular semi-positive) Hermitian metric (see \cite{Dem90}). 
The notion of bigness defined above was later shown to admit the same characterization on compact complex manifolds without assuming projectivity (see \cite[Theorem 2.3.30]{MM07}), by combining Demailly's approximation with the holomorphic Morse inequalities. 
However, when the space has singularities, it seems to be unknown whether bigness and singular positivity are equivalent.

\begin{theorem}[{\cite[Theorem 3.3]{Wat26b}}]\label{Theorem: characterization of big and singular positive}
    Let $X$ be a compact complex space and $L\longrightarrow X$ be a holomorphic line bundle. 
    We have the following relationship.
    \begin{itemize}
        \item If $L$ is singular positive, then $L$ is big. 
        \item If $L$ is big, then there exists a singular Hermitian metric $h$ on $L$ such that $\iO{L,h}\geq\gamma$ on $\reg$ in the sense of currents for some Hermitian metric $\gamma$ on $X$.
        
        Furthermore, if $X$ is normal, then we can take the singular Hermitian metric $h$ with $\iO{L,h}\geq\gamma$ on $X$ in the sense of currents 
        on $X$, i.e., $h$ is singular positive. 
    \end{itemize}
    Thus, if $X$ is normal, then singular positivity and bigness coincide.
\end{theorem}

\subsection{Logarithmic $L^2$-Dolbeault resolusion and isomorphism}

In this subsection, we introduce the logarithmic $L^2$-Dolbeault resolution on complex manifolds. 
The analytic logarithmic $L^2$-Dolbeault resolution for line bundles on compact complex manifolds was established in \cite{HLWY23}. 
By the same argument, we can extend it to holomorphic vector bundles over relatively compact subsets as follows.

Let $M$ be a complex manifold, $D$ be a simple normal crossing divisor and $\omega_D$ be a smooth \kah metric on $M\setminus D$. 
Let $h_F$ be a smooth Hermitian metric on a vector bundle $F|_{M\setminus D}$. The sheaf $\scr{L}^{p,q}_{F,h_F,\,\omega_D}$ over $M$ is defined as follows: 
For any open subset $U$ of $M$, the section space $\scr{L}^{p,q}_{F,h_F,\,\omega_D}(U)$ consists of $F$-valued $(p,q)$-forms $u$ witn measurable coefficients such that $|u|^2_{h_F,\,\omega_D}dV_{\omega_D}$ and $|\dbar u|^2_{h_F,\,\omega_D}dV_{\omega_D}$ are integrable on $U\setminus D$. 

\begin{theorem}[{cf. \cite[Theorem 3.1]{HLWY23}}]\label{Theorem: Logarithmic L2-Dolbeault resolusion}
    Let $Y$ be a complex manifold, not necessarily compact. Let $X\Subset Y$ be a relatively compact open subset, and let $\omega$ be a \kah metric on $X$. 
    Let $D=\sum^s_{j=1} D_j$ be a simple normal crossing divisor on a neighborhood of $\overline{X}$, and let $\omega_P$ be a smooth \kah metric on $X\setminus D$ which is of Poincar\'{e}-type along $D$. 
    Let $F$ be a holomorphic vector bundle on $X$ with a smooth Hermitian metric $h_F$ and $\sigma_j$ be the defining section of $D_j$. Fix smooth Hermitian metrics $h_j$ on $\cal{O}_X(D_j)$.  
    
    Then, there exists a sufficiently large integer $\alpha>0$ depending on the compactness of $X$ such that, for any fixed constants $\tau_j\in (0,1]$, the smooth Hermitian metric $h_{\alpha,\tau_j}$ on $F|_{X\setminus D}$ defined by 
    \begin{align*}
        h^F_{\tau,\alpha}=h_F\prod^s_{j=1}|\sigma_j|^{2\tau_j}_{h_j}(\log |\sigma_j|^2_{h_j})^{2\alpha}
    \end{align*}
    has the property that, for any $p\geq0$, the logarithmic $L^2$-Dolbeault complex 
    \begin{align*}
        0\longrightarrow \Omega_X^p(\log D)\otimes\cal{O}_X(F)\hookrightarrow\scr{L}^{p,0}_{F,h^F_{\tau,\alpha},\,\omega_P}\overset{\dbar}{\longrightarrow}\scr{L}^{p,1}_{F,h^F_{\tau,\alpha},\,\omega_P}\overset{\dbar}{\longrightarrow}\scr{L}^{p,2}_{F,h^F_{\tau,\alpha},\,\omega_P}\overset{\dbar}{\longrightarrow}\cdots
    \end{align*}
    is exact on $X$; that is, the complex $(\scr{L}^{p,\ast}_{F,h^F_{\tau,\alpha},\,\omega_P},\dbar)$ is logarithmic $L^2$-Dolbeault fine resolution of $\Omega_X^p(\log D)\otimes\cal{O}_X(F)$.
    Thus, we have the logarithmic $L^2$-Dolbeault isomorphism 
    \begin{align*}
        H^q(X,\Omega_X^p(\log D)\otimes F)\cong H^q\Bigl(\Gamma\bigl(X\setminus D,\mathscr{L}^{p,\ast}_{F,h^F_{\tau,\alpha},\,\omega_P}\bigr)\Bigr)
    \end{align*}
    for any $p,q\geq0$.
\end{theorem}

If $h_F$ is a singular Hermitian metric, then, assuming that $h_F$ has an appropriate (quasi-)semi-positivity property, such as Griffiths or Nakano positivity, 
by taking $\tau_j=1$, we can similarly obtain a logarithmic $L^2$-resolution for $\Omega_X^p(\log D)$ twisted by the $L^2$-subsheaf $\scr{E}(h_F)$ with respect to the singular Hermitian metric $h^F_{\alpha,\tau}$ constructed in the same way (see \cite[Theorems 3.2 and 3.7]{Wat26a}). 
Furthermore, if there exists a constant $0<\delta\leq1$ such that $\nu(\det h_F,x)<\delta/2$ for all points $x\in D$, then we can take $\tau_j\in(\delta,1]$.

\section{Refined Demailly's approximation and positive $\bb{Q}$-line bundles}

In this section, we prove Theorem \ref{Theorem: positivity of Q-line bundle in Introduction} by using the following refined version of Demailly's approximation, 
obtained via the strong openness property, which preserves the multiplier ideal sheaf and has algebraic singularities. 

\begin{theorem}[{=\,\cite[Theorem\,3.2]{Wat24}}]\label{Theorem: Demailly approximation with alg sing and ideal sheaves}
    Let $X$ be a complex manifold equipped with a Hermitian metric $\omega$ and $T=\alpha+\idd\varphi$ be a closed $(1,1)$-current on $X$ 
    where $\alpha$ is a smooth closed $(1,1)$-form and $\varphi$ is a quasi-plurisubharmonic function. Assume that $T=\alpha+\idd\varphi\geq \gamma$ holds for a continuous real $(1,1)$-form $\gamma$ on $X$. 
    Then, for a relatively compact subset $K\Subset X$, there exist an increasing sequence of positive integers $\{m_\nu\}_{\nu\in\mathbb{N}}$ and a sequence of quasi-plurisubharmonic functions $\{\varphi_{m_\nu}\}_{\nu\in\bb{N}}$ on $K$ such that the following are satisfied.
    \begin{itemize}
        \item [$(a)$] $\varphi_{m_\nu}$ has algebraic singularities. 
        That is, locally $\varphi_{m_\nu}$ can be expressed as
            \begin{align*}
                \varphi_{m_\nu}=\frac{1+2^{-\nu}}{2m_\nu}\log\sum_j|\sigma_{m_\nu,j}|^2+\varPhi_{m_\nu},
            \end{align*}
            where $\{\sigma_{m_\nu,j}\}_{j\in\bb{N}}$ is an orthonormal basis of the Hilbert space $\cal{H}(m_\nu(\varphi+c|z|^2))$ for some real number $c$ and local coordinates $z$ and $\varPhi_{m_\nu}$ is smooth. 
        \item [$(b)$] $\varphi_{m_\nu}$ is smooth on $K\setminus Z_\nu$, where the set $Z_\nu$ of logarithmic poles of $\varphi_{m_\nu}$ is an analytic subset of $K$, satisfying $Z_\nu\subset Z_{\nu+1}$, and is locally obtained by $\bigcap_j \sigma_{m_\nu,j}^{-1}(0)$. 
        \item [$(c)$] there exists a large $\nu_0$ such that $\scr{I}(\varphi)=\scr{I}(\varphi_{m_\nu})$ on $K$ for any $\nu\geq \nu_0$. 
        \item [$(d)$] $T_{m_\nu}:=\alpha+\idd\varphi_{m_\nu}$ satisfies $T_{m_\nu}\geq\gamma-\varepsilon_\nu\omega$, where $\{\varepsilon_\nu\}_{\nu_0\leq\nu\in\bb{N}}$ is a decreasing sequence of positive numbers with $\lim_{\nu\to+\infty}\varepsilon_\nu=0$.
        \item [$(e)$] we obtain the following inequality related to Lelong numbers 
        \begin{align*}
            \nu(\varphi,x)-\frac{n}{m_\nu}\leq\nu\Big(\frac{\varphi_{m_\nu}}{1+2^{-\nu}},x\Big)\leq\nu(\varphi,x) \quad\text{ for any } x\in K.
        \end{align*}
    \end{itemize}

    In particular, for any $t>0$ there exists an enough large integer $\nu(t)\in\bb{N}$ such that $\scr{I}(t\varphi)=\scr{I}(t\varphi_{m_\nu})$ for any $\nu\geq\nu(t)$.
\end{theorem}

\begin{theorem}\label{Theorem: Blow-ups and Q-line bundles}
    Let $X$ be a complex manifold equipped with a Hermitian metric $\omega$ and $L\longrightarrow X$ be a holomorphic line bundle equipped with a singular Hermitian metric $h$. Let $V$ be a relatively compact open subset of $X$.
    If $h$ is singular positive on an open neighborhood of $\overline{V}$, i.e., $\iO{L,h}\geq\varepsilon\omega$ on an open neighborhood of $\overline{V}$ in the sense of currents for a smooth positive function $\varepsilon:X\longrightarrow\bb{R}_{>0}$, then 
    there exists a singular Hermitian metric $\hbar$ with algebraic singularities given by a refined Demailly approximation of $h$ in Theorem \ref{Theorem: Demailly approximation with alg sing and ideal sheaves} such that the following conditions are satisfied. 
    \begin{itemize}
        \item [$(a)$] $\scr{I}(h)=\scr{I}(\hbar)$ on $V$.
        \item [$(b)$] $\iO{L,\hbar}\geq\varepsilon\omega/2$ in the sense of currents on $V$. 
    \end{itemize}
    
    Let $Z$ be the analytic subset of $V$ given by the singular locus of $\hbar$ and $\pi:\tl{V}\longrightarrow V$ be a log resolution of $Z$ such that $E=\sum_{j\in J}E_j:=\pi^{-1}(Z)_{red}$ is a simple normal crossing divisor. 
    Then, the following conditions hold:
    \begin{itemize}
        \item [$(\alpha)$] for divisorial Lelong numbers $\upsilon_{E_j}(\hbar):=\nu(\pi^*\hbar,E_j)$, we obtain 
        \begin{align*}
            \scr{I}(\pi^*\hbar)=\cal{O}_{\tl{V}}\Big(\!-\sum_{j\in J}\lfloor\upsilon_{E_j}(\hbar)\rfloor E_j\Big).
        \end{align*}
        \item [$(\beta)$] there exist nonnegative rational numbers $\delta_j\in[\upsilon_{E_j}(\hbar)-\lfloor\upsilon_{E_j}(\hbar)\rfloor,1)\cap\bb{Q}\subset[0,1)$, $j\in J$, such that the following $\bb{Q}$-line bundle is positive on $\tl{V}$. 
        \begin{align*}
            \qquad\qquad\pi^*L\otimes\scr{I}(\pi^*\hbar)\otimes\cal{O}_{\tl{V}}\Big(\!-\sum_{j\in J}\delta_jE_j\Big)=\pi^*L\otimes\cal{O}_{\tl{V}}\Big(\!-\sum_{j\in J}(\lfloor\upsilon_{E_j}(\hbar)\rfloor+\delta_j)E_j\Big)
        \end{align*}
    \end{itemize}
\end{theorem}

\begin{proof}
    For a smooth Hermitian metric $h_0$ on $L$, there exists a locally integrable function $\varphi\in L^1_{loc}(X,\bb{R})$ such that $h$ can be written as $h=h_0e^{-2\varphi}$ on $X$ and $\varphi$ is quasi-plurisubharmonic by $\iO{L,h}\geq\varepsilon\omega$. 
    For this $\varphi$, from Theorem \ref{Theorem: Demailly approximation with alg sing and ideal sheaves}, there exist an increasing sequence of positive integers $\{m_\nu\}_{\nu\in\bb{N}}$ and a sequence of quasi-plurisubharmonic functions $\{\varphi_{m_\nu}\}_{\nu\in\bb{N}}$ that satisfy conditions $(a)$-$(e)$ of Theorem \ref{Theorem: Demailly approximation with alg sing and ideal sheaves}. 
    Furthermore, there exists a sufficiently large integer $\nu_0$ such that, for every $\nu\geq\nu_0$, the singular Hermitian metric $h_\nu:=h_0e^{-2\varphi_{m_\nu}}$ can be chosen to satisfy $(a)$ and $(b)$. Throughout the remainder of the proof, we fix such a $\nu$ with $\nu\geq\nu_0$ and write $\varphi_\nu:=\varphi_{m_\nu}$ and $\hbar:=h_\nu=he^{-2\varphi_\nu}$. 
     
    As in the proof of \cite[Theorem 3.5]{Wat24}, we introduce the ideal $\scr{J}_\nu$ of germs of holomorphic function $f$ such that $|f|\leq C \exp{\frac{m_\nu\varphi_\nu}{1+2^{-\nu}}}$ for some constant $C$. 
    This is a globally defined ideal sheaf. Under the notation of the proof of Theorem \ref{Theorem: Demailly approximation with alg sing and ideal sheaves} (= \cite[Theorem 3.2]{Wat24}),
    from the strong Noetherian property of coherent ideal sheaves (see \cite[ChapterII,\,(3.22)]{Dem-book}), the sequence of ideal sheaves generated by the holomorphic functions $\{\sigma_{\nu,j,k}(z)\overline{\sigma_{\nu,j,k}(\overline{w})}\}_{k\leq K_j}$ on $B_j\times B_j$ is locally stationary as $K_j$ increases, hence independant of $K_j$ on $B'_j\times B'_j\Subset B_j\times B_j$ for $K_j$ large enough.
    By uniform convergence of $\sum^{\infty}_{k=1}\sigma_{\nu,j,k}(z)\overline{\sigma_{\nu,j,k}(\overline{w})}$ on compact sets, this sum of the series is a section of the coherent ideal sheaf generated by $\{\sigma_{\nu,j,k}(z)\overline{\sigma_{\nu,j,k}(\overline{w})}\}_{k\leq K_j}$ over $B'_j\times B'_j$.
    Hence, for some $C_j>0$ we obtain (see the proof of \cite[Theorem\,2.2.1]{DPS01}) 
    \begin{align*}
        \exp\frac{m_\nu(\varphi_{\nu}-\varPhi_{m_\nu})}{1+2^{-\nu}}=\sum^{+\infty}_{k=1}|\sigma_{\nu,j,k}(z)|^2\leq C_j\sum^{K_j}_{k=1}|\sigma_{\nu,j,k}(z)|^2 \quad\text{on}\quad B'_j,
    \end{align*}
    where $e^{-\varPhi_{m_\nu}}$ is bounded. 
    Then, $\scr{J}_\nu$ locally equal to the integral closure $\overline{\cal{J}_j}$ of the ideal sheaf $\cal{J}_j=(\sigma_{\nu,j,1},\ldots,\sigma_{\nu,j,K_j})$ by Brian\c{c}on-Skoda's theorem (see \cite[Theorem\,(11.17)]{Dem12}), 
    thus $\scr{J}_\nu$ is coherent. Furthermore, we have 
    \begin{align*}
        \mathrm{supp}\,\mathcal{O}_V/\!\scr{J}_\nu=V(\scr{J}_\nu)=V(\overline{\cal{J}_j})=V(\!\sqrt{\overline{\cal{J}_j}})=V(\sqrt{\cal{J}_j})=V(\cal{J}_j)
        =\!\!\bigcap_{1\leq k\leq K_j}\!\!\sigma_{\nu,j,k}^{-1}(0)=Z_\nu
    \end{align*}
    on each $B'_j$, and hence $\mathrm{supp}\,\mathcal{O}_V/\scr{J}_\nu=Z_\nu$ on a neighborhood of $V$, where $V\Subset\bigcup_j B'_j$ is a finite open cover of $V$.
    For simplicity, we set $Z:=Z_\nu$.

    By Hironaka's desingularization theorem \cite{Hir64}, there exists a log resolution $\pi:\widetilde{V}\longrightarrow V$ of $Z$ obtained by a finite sequence of blow-ups with smooth centers 
    such that $E=\sum_{\ell\in J}E_\ell:=\pi^{-1}(Z)_{red}$ is a simple normal crossing divisor.    
    For simplicity, we omit \( j \) locally and denote the local orthonormal basis by $\{\sigma_{\nu,k}\}_{k\in\bb{N}}$.
    Let $g$ be the local generator of the ideal generated by $\{\pi^*\sigma_{\nu,k}\}_{k\in\bb{N}}$, then
    there exists holomorphic functions $\tau_{\nu,k}$ such that $\pi^*\sigma_{\nu,k}=g\cdot \tau_{\nu,k}$, where $\tau_{\nu,k}$ have no common zeros (see \cite[Lemma\,2.3.19]{MM07}). 
    We consider the decomposition $g=\prod g_j^{a_j}$ of $g$ in irreducible factors, 
    then the local weight of $\pi^*\hbar$ has the form 
    \begin{align*}
        \pi^*\varphi_{\nu}=\frac{1+2^{-\nu}}{m_\nu}\sum_{j\in J} a_j\log|g_j|+\frac{1+2^{-\nu}}{2m_\nu}\log\sum_{k\in\bb{N}}|\tau_{\nu,k}|^2+\pi^*\varPhi_{m_\nu},
    \end{align*}
    where $\pi^*\hbar=\pi^*h_0\cdot e^{-2\pi^*\varphi_{\nu}}$ on $\widetilde{V}$. 
    After reordering the indices, we have that each $E_j$ is the global divisors given by the defining sections $g_j$.  
    For simplicity, we set $\upsilon_j:=(1+2^{-\nu})\frac{a_j}{m_\nu}\in\bb{Q}_{>0}$ for each $j\in J$. Then, the above expression for $\pi^*\varphi_\nu$ implies that 
    \begin{align*}
        \upsilon_{E_j}(\hbar):=\nu(\pi^*\hbar,E_j)=\nu(\pi^*\varphi_\nu,E_j):=\inf_{y\in E_j}\nu(\pi^*\varphi_\nu,y)=\nu(\pi^*\varphi_\nu,\eta_j)=\frac{1+2^{-\nu}}{m_\nu}a_j=\upsilon_j,
    \end{align*}
    where $\eta_j$ is a generic point of $E_j$ (see \cite[Chapter III, Lemma 8.15]{Dem-book}). 
    Furthermore, it is already known that (see \cite[Remark\,5.9]{Dem12})
    \begin{align*}
        \scr{I}(\pi^*\hbar)=\cal{O}_{\widetilde{V}}\Bigl(-\sum_{j\in J}\lfloor\frac{(1+2^{-\nu})}{m_\nu}a_j\rfloor E_j\Bigr)=\cal{O}_{\widetilde{V}}\Bigl(-\sum_{j\in J}\lfloor\upsilon_j\rfloor E_j\Bigr),
    \end{align*}
    and $\pi^*L\otimes\scr{I}(\pi^*\hbar)=\pi^*L\otimes\cal{O}_{\widetilde{V}}(-\sum_{j\in J}\lfloor\upsilon_j\rfloor E_j)$ is a holomorphic line bundle.

    Define a natural singular Hermitian metric on the $\bb{Q}$-line bundle $\cal{O}_{\widetilde{V}}\big(\sum_{j\in J}\upsilon_jE_j\big)$ by 
    \begin{align*}
        h_\upsilon:=\frac{1}{\prod_{j\in J}|g_j|^{2\upsilon_j}},        
    \end{align*}
    which has the same singularities as $\pi^*\hbar$. By canceling the singularities, the Hermitian metric $\pi^*\hbar\otimes h^*_\upsilon$ on the $\bb{Q}$-line bundle 
    \begin{align*}
        \cal{L}_\upsilon:=\pi^*L\otimes\cal{O}_{\widetilde{V}}\Bigl(-\sum_{j\in J}\upsilon_jE_j\Bigr)=\pi^*L\otimes\scr{I}(\pi^*\hbar)\otimes\cal{O}_{\widetilde{V}}\Bigl(-\sum_{j\in J}(\upsilon_j-\lfloor\upsilon_j\rfloor)E_j\Bigr)
    \end{align*}
    is smooth. Since $\iO{\pi^*L,\pi^*\hbar}\geq\pi^*(\varepsilon\omega)/2$ on $\tv$ in the sense of currents, the smooth curvature form of $\pi^*\hbar\otimes h^*_\upsilon$ on $\cal{L}_\upsilon$ also satisfies $\iO{\cal{L}_\upsilon,\pi^*\!\hbar\,\otimes h^*_\upsilon}\geq\pi^*(\varepsilon\omega)/2$ on $\tv$.
    Indeed, $\iO{\cal{O}_{\widetilde{V}}(\sum \upsilon_jE_j),h_\upsilon}=\sum_{j\in J}\upsilon_j[E_j]\geq0$ in the sense of currents, and it vanishes on $\tv\setminus E$, where $[E_j]$ denotes the current of integration along $E_j$.
    Similarly, $\iO{\pi^*L,\pi^*\hbar}$ is smooth on $\tv\setminus E$. We first have $\iO{\cal{L}_\upsilon,\pi^*\hbar\otimes h^*_\upsilon}\geq\pi^*(\varepsilon\omega)/2$ on $\tv\setminus E$. 
    Since the curvature form $\iO{\cal{L}_\upsilon,\pi^*\hbar\otimes h^*_\upsilon}$ is smooth, the inequality extends to all of $\tv$. 

    The $\pi$-exceptional divisor $\exc(\pi)$ is also a simple normal crossing divisor, and there exists a subset $J'\subseteq J$ such that $\exc(\pi)=\sum_{j\in J'}E_j\subseteq E$.
    Here, $\pi^*(\varepsilon\omega)$ is positive on $\tv\setminus\exc(\pi)$, while its positivity degenerates along $\exc(\pi)$, and hence it is only semi-positive on $\tv$. 
    By applying the Negativity Lemma (see \cite[Remark 1.6.2 (2)]{Kaw24}, \cite[Lemma 2.2]{Wat25b}), there exist integers $\kappa_j\in\bb{N}\cup\{0\}$ such that the holomorphic line bundle $P_\kappa:=\cal{O}_{\tv}\big(-\sum_{j\in J}\kappa_jE_j\big)$ admits a smooth Hermitian metric $\cal{H}_\kappa$ whose curvature compensates for the degeneration of the positivity of $\pi^*\omega$ along $\exc(\pi)$. 
    Here, the support of $\sum_{j\in J}\kappa_jE_j$ coincides with $\exc(\pi)$, i.e., $\kappa_j\in\bb{N}$ for $j\in J'$ and $\kappa_j=0$ for $j\in J\setminus J'$.
    By the compactness of $V$, there exists a sufficiently large $m\in\bb{N}$ such that the smooth curvature $\iO{P_{\kappa},\cal{H}_\kappa}+m\pi^*(\varepsilon\omega)/2$ is positive on $\tv$.
    Let 
    \begin{align*}
        \ell:=\max_{j\in J}\Bigl\{\frac{\kappa_j}{1-\upsilon_j+\lfloor\upsilon_j\rfloor}-m,0\Bigr\}+1\in\bb{Q}_{\geq1}. 
    \end{align*}
    Then, for each $j\in J$, we obtain 
    \begin{align*}
        0\leq\vartheta_j:=\frac{\kappa_j}{m+\ell}< 1-\upsilon_j+\lfloor\upsilon_j\rfloor.
    \end{align*}
    Furthermore, it is clear that $\iO{P_{\kappa},\cal{H}_\kappa}+(m+\ell)\pi^*(\varepsilon\omega)/2>0$ on $\tv$.
    After dividing by $m+\ell$, the smooth Hermitian metric $\cal{H}_\kappa^{1/(m+\ell)}$ on the $\bb{Q}$-line bundle $P_\kappa^{1/(m+\ell)}=\cal{O}_{\tv}(-\sum_{j\in J}\vartheta_j E_j)$ satisfies $\iO{P_{\kappa}^{1/(m+\ell)},\cal{H}_\kappa^{1/(m+\ell)}}+\pi^*(\varepsilon\omega)/2>0$ on $\tv$.
    
    Setting $\delta_j:=\vartheta_j+\upsilon_j-\lfloor\upsilon_j\rfloor$, we obtain $\delta_j\in[\upsilon_j-\lfloor\upsilon_j\rfloor,1)\cap\bb{Q}\subset[0,1)$ for every $j\in J$. 
    Consequently, the Hermitian metric $\pi^*\hbar\otimes h^*_\upsilon\otimes\cal{H}_\kappa^{1/(m+\ell)}$ on the $\bb{Q}$-line bundle 
    \begin{align*}
        \cal{L}_\upsilon\otimes P_\kappa^{\frac{1}{m+\ell}}&=\pi^*L\otimes\scr{I}(\pi^*\hbar)\otimes\cal{O}_{\widetilde{V}}\Bigl(-\sum_{j\in J}(\upsilon_j-\lfloor\upsilon_j\rfloor)E_j\Bigr)\otimes\cal{O}_{\widetilde{V}}\Bigl(-\sum_{j\in J}\vartheta_jE_j\Bigr)\\
        &=\pi^*L\otimes\scr{I}(\pi^*\hbar)\otimes\cal{O}_{\widetilde{V}}\Bigl(-\sum_{j\in J}\delta_jE_j\Bigr)
    \end{align*}
    is smooth and has positive curvature on $\tv$. Indeed, we obtain the curvature inequality
    \begin{align*}
        \iO{\cal{L}_\upsilon,\pi^*\!\hbar\,\otimes h^*_\upsilon}+\frac{1}{m+\ell}\iO{P_\kappa,\cal{H}_\kappa}\geq\frac{\pi^*(\varepsilon\omega)}{2}+\frac{1}{m+\ell}\iO{P_\kappa,\cal{H}_\kappa}>0
    \end{align*}
    on $\tv$. Thus, the proof is complete.
\end{proof}

Theorem \ref{Theorem: positivity of Q-line bundle in Introduction} follows immediately from Theorem \ref{Theorem: Blow-ups and Q-line bundles} and Demailly's characterization (see \cite{Dem90}, \cite[Theorem 2.3.30]{MM07}).

\begin{theorem}\label{Theorem: positivity of Q-line bundles for nef and big}
    Let $X$ be a compact complex manifold equipped with a Hermitian metric $\omega$ and $L\longrightarrow X$ be a holomorphic line bundle. 
    If $L$ is nef and big, then there exists a singular Hermitian metric $h_{nb}$ with algebraic singularities, whose curvature is a strictly positive current, such that $\scr{I}(h_{nb})=\cal{O}_X$.
    Let $Z$ be the analytic subset of $X$ given by the singular locus of $h_{nb}$ and $\pi:\tx\longrightarrow X$ be a log resolution of $Z$ such that $E=\sum_{j\in J}E_j$ $=\pi^{-1}(Z)_{red}$ is a simple normal crossing divisor. 
    Then, the following conditions hold:
    \begin{itemize}
        \item [$(\alpha)$] for every $j\in J$, the divisorial Lelong number $\upsilon_{E_j}(h_{nb})$ satisfies $\upsilon_{E_j}(h_{nb})<1$. Therefore, we obtain $\scr{I}(\pi^*h_{nb})=\cal{O}_{\tx}$.
        \item [$(\beta)$] there exist positive rational numbers $\delta_j\in[\upsilon_{E_j}(h_{nb}),1)\cap\bb{Q}\subset(0,1)$, $j\in J$, such that the $\bb{Q}$-line bundle $\pi^*L\otimes\cal{O}_{\tx}\big(\!-\sum_{j\in J}\delta_jE_j\big)$ is positive on $\tx$. 
    \end{itemize}
\end{theorem}

\begin{proof}
    From the bigness of $L$, there exists a singular Hermitian metric $h$ on $L$ with a strictly positive curvature current (see \cite{Dem90}, \cite[Theorem 2.3.30]{MM07}). 
    By the compactness of $X$, there exists a sufficiently small positive rational number $c_0\in(0,1)\cap\bb{Q}$ such that $\iO{L,h}\geq 2c_0\omega$ holds on $X$ in the sense of currents. 
    For the singular Hermitian metric $h$, as in the proof of Theorem \ref{Theorem: Blow-ups and Q-line bundles}, we take a singular Hermitian metric $\hbar$, its singular locus $Z$, and a log resolution $\pi:\tx\longrightarrow X$ of $Z$. 
    Then, we obtain $\iO{L,h}\geq c_0\omega$ on $X$ and $\iO{\pi^*L,\pi^*\hbar}\geq c_0\pi^*\omega$ on $\tx$ in the sense of currents. Set $\upsilon_j:=\upsilon_{E_j}(\hbar)$. 

    Since $L$ is nef, for every sufficiently small positive number $\varepsilon>0$, there exists a smooth Hermitian metric $h_\varepsilon$ on $L$ such that $\iO{L,h_\varepsilon}\geq-\varepsilon\omega$ on $X$. 
    In particular, we take $\varepsilon$ to be a rational number in what follows. Let $\upsilon_{max}:=\max_{j\in J}\upsilon_j$ and $\nu_{max}:=\max_{x\in X}\nu(\hbar,x)$. 
    Take and fix a sufficiently small positive rational number $\varepsilon$ satisfying $0<\varepsilon<c_0/\max\{\upsilon_{max},\nu_{max}\}$. Using this $\varepsilon\in\bb{Q}_{>0}$, define a new singular Hermitian metric $h_{nb}$ on $L$ by the convex combination $h_{nb}:=\hbar^{\varepsilon/c_0}h_\varepsilon^{1-\varepsilon/c_0}$.  
    The metric $h_{nb}$ clearly has algebraic singularities with singular locus $Z$, which is the same as that of $\hbar$.
    The following inequality shows that the curvature current of $h_{nb}$ is strictly positive.
    \begin{align*}
        \iO{L,h_{nb}}=\frac{\varepsilon}{c_0}\iO{L,\hbar}+\Bigl(1-\frac{\varepsilon}{c_0}\Bigr)\iO{L,h_\varepsilon}\geq\varepsilon\omega-\Bigl(1-\frac{\varepsilon}{c_0}\Bigr)\varepsilon\omega=\frac{\varepsilon^2}{c_0}\omega>0.
    \end{align*}
    By a result of Skoda (see \cite[Lemma 5.6]{Dem12}), since $\nu(h_{nb},x)=\frac{\varepsilon}{c_0}\nu(\hbar,x)\leq\frac{\varepsilon\nu_{max}}{c_0}<1$ for every $x\in X$, we obtain $\scr{I}(h_{nb})=\cal{O}_X$, and we also obtain the desired estimate
    \begin{align*}
        \upsilon_{E_j}(h_{nb}):=\nu(\pi^*h_{nb},E_j)=\frac{\varepsilon}{c_0}\nu(\pi^*\hbar,E_j)=\frac{\varepsilon}{c_0}\upsilon_{E_j}(\hbar)=\frac{\varepsilon}{c_0}\upsilon_j<1
    \end{align*}
    for the divisorial Lelong numbers for every $j\in J$. 

    Finally, by applying the Negativity Lemma and arguing as in the proof of Theorem \ref{Theorem: Blow-ups and Q-line bundles}, we take smooth Hermitian metrics $\cal{H}_\kappa$ on $P_\kappa$ to compensate for the positivity of $\pi^*\omega$, 
    and by choosing $m$ and $\ell$ (equivalently, $\delta_j\in\bb{Q}_{>0}$) appropriately, we can construct a smooth Hermitian metric satisfying the final condition $(\beta)$.
\end{proof}

\begin{conjecture}
    For a singular Hermitian metric $h$ on $L$, can the metric $\hbar$ in Theorem \ref{Theorem: Blow-ups and Q-line bundles} be chosen appropriately, using condition $(e)$ in Theorem \ref{Theorem: Demailly approximation with alg sing and ideal sheaves} on the Lelong numbers, to satisfy $\scr{I}(\pi^*h)=\scr{I}(\pi^*\hbar)$?
\end{conjecture}

\section{Steenbrink-type vanishing for singular positive line bundles}

In this section, we establish various Steenbrink-type vanishing theorems for singular positive line bundles. 
First, as an extension of Theorem \ref{Theorem: Steenbrink vanishing for big in Introduction}, we present the following result for relatively compact weakly pseudoconvex complex spaces.
Here, a function $\varPsi:X\longrightarrow[-\infty,+\infty)$ on a complex space $X$ is \textit{exhaustion} if all sublevel sets $X_c:=\{x\in X\mid\varPsi(x)<c\}$, $\forall\,c\in\bb{R}$, are relatively compact. 
A complex space is said to be \textit{weakly} \textit{pseudoconvex} if there exists a smooth exhaustion plurisubharmonic function.

\begin{theorem}\label{Theorem: Steenbrink vanishing for singular positive in not Introduction}
    Let $Y$ be a complex space of pure dimension $n$ and $\mu:\widehat{Y}\longrightarrow Y$ be a resolusion of singularities. 
    Let $X$ be a relatively compact weakly pseudoconvex open subset of $Y$, let $V$ be an open neighborhood of $\overline{X}$, and let $L\longrightarrow V$ be a holomorphic line bundle with a singular Hermitian metric $h$. 
    
    If $h$ is singular positive on $V$, then there exists a singular Hermitian metric $\hbar$ on $L|_X$ such that $\mu^*\hbar$ is singular positive on $\hx:=\mu^{-1}(X)$, has algebraic singularities on $\hx\setminus\exc(\mu)$, and satisfies $\scr{I}(\mu^*h)=\scr{I}(\mu^*\hbar)$ on $\hx$. 
    For a log resolusion $\pi:\tx\longrightarrow\hx$ of the singular locus $Z$ of $\mu^*\hbar$, the following Steenbrink-type vanishing holds: 
    \begin{align*}
        H^q(\tx,\Omega^p_{\tx}(\log E)\otimes\cal{O}_{\tx}(-E)\otimes\tl{\pi}^*L\otimes\scr{I}(\tl{\pi}^*\hbar))=0,
    \end{align*}
    for any $p+q>n$, where $\tx:=\pi^{-1}(\hx)$ and $\tl{\pi}:=\mu\circ\pi:\tx\longrightarrow X$, and $E:=\pi^{-1}(Z)_{red}$ is a simple normal crossing divisor. 
\end{theorem}

\begin{proof}
    After shrinking $V$, we may assume that $V$ is a relatively compact open neighborhood of $\overline{X}$, and set $\wh{V}:=\mu^{-1}(V)$.
    Arguing as in the proofs of \cite[Theorem 1.5]{Wat26b} and \cite[Theorem 4.17]{Wat26c}, and using \cite[Lemma 3.2]{Wat25b} together with the strong openness property (see \cite{GZ15}), there exist a quasi-plurisubharmonic function $\psi:\wh{V}\longrightarrow[-\infty,+\infty)$ which is smooth on $\wh{V}\setminus\exc(\mu)$ 
    and a sufficiently small constant $\varepsilon_V>0$ such that the singular Hermitian metric $\cal{H}:=\mu^*he^{-\varepsilon_V\psi}$ on $\mu^*L|_{\wh{V}}$ is singular positive on $\wh{V}$ and satisfies $\scr{I}(\mu^*h)=\scr{I}(\cal{H})$ on $\wh{V}$. 
    
    By the refined Demailly approximation Theorem \ref{Theorem: Demailly approximation with alg sing and ideal sheaves}, there exists a singular Hermitian metric $\wh{\cal{H}}$ on $\mu^*L|_{\wh{V}}$ such that $\wh{\cal{H}}$ is singular positive on $\wh{V}$, has algebraic singularities, and satisfies $\scr{I}(\wh{\cal{H}})=\scr{I}(\cal{H})$ on $\wh{V}$. 
    Let $Z$ be the analytic subset of $\hv$ given by the singular locus of $\wh{\cal{H}}$ and $\pi:\tv\longrightarrow\hv$ be a log resolusion of $Z$ such that $E=\sum_{j\in J}E_j:=\pi^{-1}(Z)_{red}$ is a simple normal crossing divisor. 
    Then, it follows from Theorem \ref{Theorem: Blow-ups and Q-line bundles} that there exist nonnegative rational numbers $\delta_j\in [\upsilon_{E_j}(\wh{\cal{H}})-\lfloor\upsilon_{E_j}(\wh{\cal{H}})\rfloor,1)\cap\bb{Q}\subset [0,1)$ such that 
    \begin{align*}
        \tl{\pi}^*L\otimes\scr{I}(\pi^*\wh{\cal{H}})\otimes\cal{O}_{\tv}\Big(-\sum_{j\in J}\delta_jE_j\Big)=\tl{\pi}^*L\otimes\cal{O}_{\tv}\Big(-\sum_{j\in J}(\lfloor\upsilon_{E_j}(\wh{\cal{H}})\rfloor+\delta_j)E_j\Big)
    \end{align*}
    is positive on $\tv$, where $\tv:=\pi^{-1}(\hv)$ and $\tl{\pi}:=\mu\circ\pi:\tv\longrightarrow V$. 

    Using the biholomorphic restriction $\mu|_{\wh{Y}\setminus\exc(\mu)}:\wh{Y}\setminus\exc(\mu)\overset{\simeq}{\longrightarrow}Y\setminus Y_{sing}=Y_{reg}$, 
    we define a new singular Hermitian metric $\hbar$ on $L|_V$ by
    \[
    \hbar :=
    \begin{cases}
    \bigl((\mu|_{\wh{Y}\setminus\exc(\mu)})^{-1}\bigr)^*\wh{\cal{H}}
    & \text{on }\, V\setminus\exc(\mu),\\
    0 & \text{on }\, \exc(\mu)\setminus\mu(Z),\\
    +\infty & \text{on }\, \exc(\mu)\cap\mu(Z).
    \end{cases}
    \]
    Since $\mu^*\hbar$ coincides with $\wh{\cal{H}}$ on $\wh{V}\setminus\exc(\mu)$, it has algebraic singularities on $\wh{V}\setminus\exc(\mu)$, is singular positive on $\hv$, and satisfies $\scr{I}(\mu^*\hbar)=\scr{I}(\wh{\cal{H}})=\scr{I}(\cal{H})=\scr{I}(\mu^*h)$ on $\hv$.
    In particular, the construction of $\hbar$ implies that its singular locus $\{x\in\hv\mid\hbar(x)=+\infty\}$ coincides with $Z$. Furthermore, the following $\bb{Q}$-line bundle 
    \begin{align*}
        \cal{L}_{\delta}:=\tl{\pi}^*L\otimes\scr{I}(\tl{\pi}^*\hbar)\otimes\cal{O}_{\tv}\Big(-\sum_{j\in J}\delta_jE_j\Big)
    \end{align*}
    is positive on $\tv$, where $\scr{I}(\tl{\pi}^*\hbar)=\scr{I}(\pi^*\wh{\cal{H}})$ since $\tl{\pi}^*\hbar$ coincides with $\pi^*\wh{\cal{H}}$ almost everywhere.
    In other words, there exist a smooth Hermitian metric $h_{\cal{L}}$ on the line bundle 
    \begin{align*}
        \cal{L}:=\tl{\pi}^*L\otimes\scr{I}(\tl{\pi}^*\hbar)\otimes\cal{O}_{\tv}(-E)=\tl{\pi}^*L\otimes\cal{O}_{\tv}\Big(-\sum_{j\in J}(\lfloor\upsilon_{E_j}(\wh{\cal{H}})\rfloor+1)E_j\Big)
    \end{align*}
    and a smooth Hermitian metric $h_{E_j}$ on $\cal{O}_{\tv}(E_j)$ such that the induced smooth Hermitian metric $h_{\cal{L}_\delta}:=h_{\cal{L}}\prod_{j\in J}h^{1-\delta_j}_{E_j}$ on $\cal{L}_\delta$ has positive curvature, that is, 
    \begin{align*}
        \iO{\cal{L}_\delta,h_{\cal{L}_\delta}}=\iO{\cal{L},h_{\cal{L}}}+i\sum_{j\in J}(1-\delta_j)\Theta_{\cal{O}_{\tv}(E_j),h_{E_j}}>0
    \end{align*}
    on $\tv$, where $1-\delta_j\in (0,1]$.
    From this $\bb{Q}$-line bundle $\cal{L}_\delta$, we can construct a positive line bundle on $\tv$. Hence, $\tv$ admits a Hodge metric $\omega$. 

    The relative compactness of $\tx$ implies that there exists a positive constant $c_0>0$ such that $\iO{\cal{L}_\delta,h_{\cal{L}_\delta}}\geq 2c_0\omega$ on $\tx$. 
    Since $X$ is weakly pseudoconvex, there exists a smooth exhaustive plurisubharmonic function $\varPsi$ on $X$, and the pullback $\tl{\varPsi}:=\tl{\pi}^*\varPsi$ is again a smooth exhaustive plurisubharmonic function on $\tx$. Therefore, $\tx$ is also weakly pseudoconvex. 
    Let $\chi\in\cal{C}^{\infty}(\bb{R},\bb{R})$ be a smooth increasing convex function satisfying $\int^{+\infty}_0\sqrt{\chi''(t)}dt=+\infty$. Then, the \kah metric $c_0\omega+\idd\chi\circ\tl{\varPsi}$ is complete on $\tx$ (see \cite[Chapter VIII, $\S$5]{Dem-book}).  
    Let $\sigma_j$ be the defining section of $E_j$. The induced smooth Hermitian metric on $\cal{L}$ over $\tx\setminus E$ is defined by 
    \begin{align*}
        h^{\cal{L}}_{\delta,\varepsilon,\alpha}:=h_{\cal{L}}\, e^{-\chi\circ\tl{\varPsi}}\prod_{j\in J}|\sigma_j|^{2(1-\delta_j)}_{h_{E_j}}\big(\log(\varepsilon|\sigma_j|^2_{h_{E_j}})\big)^{2\alpha},
    \end{align*}
    where $\alpha>0$ is a sufficiently large constant. A straightforward computation shows that
    \begin{align*}
        \iO{\cal{L},\hldea}=&\,\iO{\cal{L},h_\cal{L}}+i\sum_{j\in J}(1-\delta_j)\Theta_{\cal{O}_{\tx}(E_j),h_{E_j}}+\idd\chi\circ\tl{\varPsi}\\
        &+i\sum_{j\in J}\frac{2\alpha\Theta_{\cal{O}_{\tx}(E_j),h_{E_j}}}{\log(\varepsilon|\sigma_j|^2_{h_{E_j}})}+i\sum_{j\in J}\frac{2\alpha\partial\log|\sigma_j|^2_{h_{E_j}}\wedge\overline{\partial}\log|\sigma_j|^2_{h_{E_j}}}{\big(\log(\varepsilon|\sigma_j|^2_{h_{E_j}})\big)^2}
    \end{align*}
    on $\tx\setminus E$. For the fixed constant $\alpha$, by choosing $\varepsilon>0$ sufficiently small, we obtain 
    \begin{align*}
        -c_0\omega\leq i\sum_{j\in J}\frac{2\alpha\Theta_{\cal{O}_{\tx}(E_j),h_{E_j}}}{\log(\varepsilon|\sigma_j|^2_{h_{E_j}})}\leq c_0\omega
    \end{align*}
    on $\tx$. Note that the constant $\varepsilon$ is thus fixed, and the choice of $\varepsilon$ depends on $\alpha$. 
    By the following curvature inequality 
    \begin{align*}
        \iO{\cal{L},\hldea}\geq\iO{\cal{L}_\delta,h_{\cal{L}_\delta}}-c_0\omega+i\sum_{j\in J}\frac{2\alpha\partial\log|\sigma_j|^2_{h_{E_j}}\wedge\overline{\partial}\log|\sigma_j|^2_{h_{E_j}}}{\big(\log(\varepsilon|\sigma_j|^2_{h_{E_j}})\big)^2}+\idd\chi\circ\tl{\varPsi}\geq c_0\omega,
    \end{align*}
    the curvature form $\iO{\cal{L},\hldea}$ is a \kah metric on $\tx\setminus E$. 
    Setting $\omega_{\tx\setminus E}:=\iO{\cal{L},\hldea}$ for simplicity, we see that the \kah metric $\omega_{\tx\setminus E}$ is of Poincar\'{e}-type along $E$ (see \cite{Zuc79,HLWY23}), and together with the $\idd\chi\circ\tl{\varPsi}$-term, it becomes complete on $\tx\setminus E$. 
    For sufficiently large $\alpha$, Theorem \ref{Theorem: Logarithmic L2-Dolbeault resolusion} yields the cohomological isomorphism 
    \begin{align*}
        H^q(\tx,\Omega^p_{\tx}(\log E)\otimes\cal{L})\cong H^q\Big(\Gamma\big(\tx\setminus E,\scr{L}^{p,\ast}_{\cal{L},\hldea,\,\omega_{\tx\setminus E}}\big)\Big).
    \end{align*}
    
    A direct computation yields the commutator identity 
    \begin{align*}
        \Bigl[\iO{\cal{L},\hldea},\Lambda_{\omega_{\tx\setminus E}}\Bigr]=(p+q-n)\,\rom{id}_{\cal{L}}
    \end{align*}
    on $\bigwedge^{p,q}T^*_{\tx\setminus E}\otimes\cal{L}$.
    For any $f\in \Gamma\big(\tx\setminus E,\scr{L}^{p,q}_{\cal{L},\hldea,\,\omega_{\tx\setminus E}}\big)$ satisfying $\dbar f=0$, the relative compactness of $\tx\setminus E$ implies that 
    \begin{align*}
        ||f||^2_{\tx\setminus E}:=\int_{\tx\setminus E}|f|^2_{\hldea,\,\omega_{\tx\setminus E}}dV_{\omega_{\tx\setminus E}}<+\infty, 
    \end{align*}
    that is, $f\in L^2_{p,q}(\tx\setminus E,\cal{L};\omega_{\tx\setminus E},\hldea)$. Since $\omega_{\tx\setminus E}$ is complete on $\tx\setminus E$, it follows from H\"{o}rmander's $L^2$-estimate (see \cite[Chapter VIII, Theorem 4.5]{Dem-book}) 
    that there exists a solution $u\in L^2_{p,q-1}(\tx\setminus E,\cal{L};\omega_{\tx\setminus E},\hldea)$ satisfying $\dbar u=f$ on $\tx\setminus E$ and $(p+q-n)||u||^2_{\tx\setminus E}\leq||f||^2_{\tx\setminus E}$. 
    In particular, $u\in \Gamma\big(\tx\setminus E,\scr{L}^{p,q-1}_{\cal{L},\hldea,\,\omega_{\tx\setminus E}}\big)$, and hence we obtain the cohomology vanishing 
    \begin{align*}
        H^q(\tx,\,&\Omega^p_{\tx}(\log E)\otimes\cal{O}_{\tx}(-E)\otimes\tl{\pi}^*L\otimes\scr{I}(\tl{\pi}^*\hbar))\\
        &=H^q(\tx,\Omega^p_{\tx}(\log E)\otimes\cal{L})\cong H^q\Big(\Gamma\big(\tx\setminus E,\scr{L}^{p,\ast}_{\cal{L},\hldea,\,\omega_{\tx\setminus E}}\big)\Big)=0,
    \end{align*}
    for any $p+q>n$. 
\end{proof}

Note that the completeness of $\omega_{\tx\setminus E}$ on $\tx\setminus E$ is essential, since H\"{o}rmander's $L^2$-estimate for $(p,q)$-forms cannot be applied without this assumption. 
In Theorem \ref{Theorem: Steenbrink vanishing for singular positive in not Introduction}, if $Y$ is nonsingular, then $\mu$ can be taken to be the identity map and $\wh{Y}=Y$. 
The same holds for Theorems \ref{Theorem: Steenbrink vanishing for big in Introduction} and \ref{Theorem: Steenbrink vanishing for nef and big}. 
Furthermore, if $Y$ is compact, we may take $Y=X$, and obtain the same statement simply for big line bundles on $X$ as follows.

\begin{corollary}\label{Corollary: Steenbrink-type vanishing on proj mfd}
    Let $X$ be a compact complex manifold and $L\longrightarrow X$ be a holomorphic line bundle. 
    If $L$ is big, that is, if $L$ admits a singular Hermitian metric $h$ with algebraic singularities and a strictly positive curvature current on $X$, then $X$ is Moishezon, 
    and for a log resolusion $\pi:\tx\longrightarrow X$ of the singular locus $Z$ of $h$, the following Steenbrink-type vanishing holds:
    \begin{align*}
        H^q(\tx,\Omega^p_{\tx}(\log E)\otimes\cal{O}_{\tx}(-E)\otimes\pi^*L\otimes\scr{I}(\pi^*h))=0, 
    \end{align*}
    for any $p+q>n$, where $E:=\pi^{-1}(Z)_{red}$ is the simple normal crossing divisor.  

    Furthermore, under the additional assumption that $L$ is nef, the singular Hermitian metric $h$ can be chosen such that $\scr{I}(\pi^*h)=\cal{O}_{\tx}$.
\end{corollary}

When $X$ is projective, the situation is somewhat simpler. Without using Demailly's approximation, we can directly construct a singular Hermitian metric $h$ with algebraic singularities and a strictly positive curvature current by Kodaira's lemma, which states that there exist a positive integer $m\in N$, an ample line bundle $A$, and an effective divisor $D$ such that $L^{\otimes m}\cong A\otimes\cal{O}_X(D)$. 
In fact, the required conditions are satisfied by taking a singular Hermitian metric $h$ on $L$ as $h:=h_A^{1/m}h_D^{1/m}$, where $h_A$ is a smooth Hermitian metric on $A$ with positive curvature and $h_D$ is the natural singular Hermitian metric $\cal{O}_X(D)$ induced by $D$ (see \cite{Dem90}).
Furthermore, by taking a finite sequence of blow-ups $\pi:\tx\longrightarrow X$ along the singular locus of $D$ such that $\pi^{-1}(D)$ is a simple normal crossing divisor, Theorem \ref{Theorem: positivity of Q-line bundle in Introduction} and Corollary \ref{Corollary: Steenbrink-type vanishing on proj mfd} follow in the same way.

Although Theorem \ref{Theorem: Steenbrink vanishing for big in Introduction} can also be proved in the same way as Theorem \ref{Theorem: Steenbrink vanishing for singular positive in not Introduction}, we give a shorter proof using Theorem \ref{Theorem: Blow-ups and Q-line bundles} and \cite[Theorem 1.1]{HLWY23}, since $X$ is compact.

\begin{proof}[Proof of Theorem \ref{Theorem: Steenbrink vanishing for big in Introduction}]
    The existence of a big line bundle on $X$ implies that $X$ is Moishezon (see \cite{Moi66}).
    Since bigness is a birational invariant, $\mu^*L$ is also big. By Demailly's characterization (see \cite[Theorem 2.3.30]{MM07}, \cite{Dem90}), $\mu^*L$ admits a singular Hermitian metric $\cal{H}$ with a strictly positive curvature current, that is, $\cal{H}$ is singular positive on $\hx$. 
    Applying Demailly's approximation theorem (see Theorem \ref{Theorem: Demailly approximation with alg sing and ideal sheaves}, \cite{DPS01}) to $\cal{H}$, we obtain a singular Hermitian metric $\wh{\cal{H}}$ on $\mu^*L$ with algebraic singularities, and $\wh{\cal{H}}$ is singular positive on $\hx$. 
    Let $Z$ be the analytic subset of $\hx$ given by the singular locus of $\wh{\cal{H}}$ and $\pi:\tx\longrightarrow\hx$ be a log resolusion of $Z$ such that $E=\sum_{j\in J}E_j:=\pi^{-1}(Z)_{red}$ is a simple normal crossing divisor. 
    Then, it follows from Theorem \ref{Theorem: Blow-ups and Q-line bundles} that there exist nonnegative rational numbers $\delta_j\in [\upsilon_{E_j}(\wh{\cal{H}})-\lfloor\upsilon_{E_j}(\wh{\cal{H}})\rfloor,1)\cap\bb{Q}\subset [0,1)$, such that 
    \begin{align*}
        \cal{L}_\delta:=\tl{\pi}^*L\otimes\scr{I}(\pi^*\wh{\cal{H}})\otimes\cal{O}_{\tx}\Big(-\sum_{j\in J}\delta_jE_j\Big)=\tl{\pi}^*L\otimes\cal{O}_{\tx}\Big(-\sum_{j\in J}(\lfloor\upsilon_{E_j}(\wh{\cal{H}})\rfloor+\delta_j)E_j\Big)
    \end{align*}
    is positive on $\tx$. Defining a singular Hermitian metric $h$ on $L$ in the same way as $\hbar$ in the proof of Theorem \ref{Theorem: Steenbrink vanishing for singular positive in not Introduction}, we see that $h$ satisfies the assumptions of the theorem and 
    \begin{align*}
        \cal{L}_\delta=\cal{O}_{\tx}(-E)\otimes\tl{\pi}^*L\otimes\scr{I}(\tl{\pi}^*h)\otimes\cal{O}_{\tx}\Big(\sum_{j\in J}(1-\delta_j)E_j\Big)
    \end{align*}
    holds, where $\scr{I}(\pi^*\wh{\cal{H}})=\scr{I}(\tl{\pi}^*h)$ and $1-\delta_j\in (0,1]$. 
    Finally, applying \cite[Theorem 1.1]{HLWY23} to the positivity of $\cal{L}_\delta$, we obtain the desired cohomology vanishing.
\end{proof}

Theorem \ref{Theorem: Steenbrink vanishing for nef and big} follows by applying Theorem \ref{Theorem: positivity of Q-line bundles for nef and big} and arguing as in the proof above. 

\begin{remark}\label{Remark: non-vanishing result}
    Theorem \ref{Theorem: logarithmic vanishing for big line bundle in [Wat26a]} cannot be extended to bidegrees $(p,q)$ with $p+q>\rom{dim}\,X$. 
    The following counterexample shows this.
\end{remark}

\begin{counterexample}
    Let $n>3$, and let $\pi:\tl{\bb{P}}^n\longrightarrow\bb{P}^n$ be the blow-up of $\bb{P}^n$ at a point $x_0$. Clearly the line bundle $\pi^*\cal{O}_{\bb{P}^n}(1)$ is nef and big.  
    Let $D$ be the strict transform of a hyperplane in $\bb{P}^n$ not passing through $x_0$. 
    For any integer $m\geq 2$, we obtain the following non-vanishing result: 
    \begin{align*}
        H^{n-1}(\tl{\bb{P}}^n,\Omega^{n-1}_{\tl{\bb{P}}^n}(\log D)\otimes\cal{O}_{\tl{\bb{P}}^n}(-D)\otimes\pi^*\cal{O}_{\bb{P}^n}(m))\ne0,
    \end{align*}
    where $D$ is a smooth divisor on $\tl{\bb{P}}^n$ with $\cal{O}_{\tl{\bb{P}}^n}(D)\cong \pi^*\cal{O}_{\bb{P}^n}(1)$, and $\pi^*\cal{O}_{\bb{P}^n}(1)|_D$ is ample. 
\end{counterexample}

\begin{proof}
    The case $m=2$ follows immediately from \cite[Counterexample 4.12]{Wat26a}, which is derived from Ramanujam's counterexample (see \cite{Ram72}, \cite[Chapter VII]{Dem-book}), and the case $m\geq3$ is obtained in the same way.
\end{proof}

Finally, by applying the Negativity Lemma (see \cite{Kaw24}, \cite[Lemma 2.2]{Wat25b}) and \cite[Theorem 1.1]{HLWY23}, we give a new short analytic proof of the Steenbrink vanishing.

\begin{proof}[Short proof of Theorem \ref{Theorem: Steenbrink vanishing [Ste85]} via analytic methods]
    Let $E=\sum_{j\in J}E_j:=\mu^{-1}(\Sigma)$ be the simple normal crossing divisor on $\tx$. 
    As in the proof of Theorem \ref{Theorem: Blow-ups and Q-line bundles}, by applying the Negativity Lemma (see \cite[Remark 1.6.2 (2)]{Kaw24}, \cite[Lemma 2.2]{Wat25b}), there exist nonnegative numbers $\delta_j\in[0,1)\cap\bb{Q}$ such that the $\bb{Q}$-line bundle 
    \begin{align*}
        \cal{A}_\delta:=\mu^*A\otimes\cal{O}_{\tx}\Big(-\sum_{j\in J}\delta_jE_j\Big)
    \end{align*}
    is positive on $\tx$. 
    Finally, applying \cite[Theorem 1.1]{HLWY23} to the positivity of $\cal{A}_\delta=\cal{O}_{\tx}(-E)\otimes\mu^*A\otimes\cal{O}_{\tx}\big(\sum_{j\in J}(1-\delta_j)E_j\big)$, we obtain the desired cohomology vanishing. 

    To prove the vanishing of higher direct image sheaves, it is enough to show that $H^q(\mu^{-1}(U),\Omega^p_{\tx}(\log E)\otimes\cal{O}_{\tx}(-E))=0$ for some sufficiently small Stein open neighborhood $U$ of each point $x\in X$ and any $p+q>n$. 
    Using a smooth strictly plurisubharmonic function as a weight, the trivial line bundle $\cal{O}_X$ becomes positive on $U$, and hence the vanishing follows from the same argument as in the proof of Theorem \ref{Theorem: Steenbrink vanishing for singular positive in not Introduction} and the above discussion. 
    Here, although $\mu^{-1}(U)$ is not necessarily Stein, it is weakly pseudoconvex, and we use the fact that $\mu^*\cal{O}_X\cong\cal{O}_{\tx}$.
\end{proof}


\vspace*{3mm}

\noindent
{\bf Acknowledgement.} 
The author is supported by Grant-in-Aid for Research Activity Start-up $\sharp$26K16989 from the Japan Society for the Promotion of Science (JSPS).







\end{document}